\documentclass[11pt]{amsart}
\usepackage[latin1]{inputenc}
\usepackage[english]{babel}
\usepackage{amsmath, amsfonts, amssymb, amsthm, mathtools, mathrsfs}
\usepackage{enumerate}
\usepackage[numbers]{natbib}
\usepackage{setspace}

\usepackage[arrow,matrix,curve]{xy}

\usepackage[pagebackref,hypertexnames=true]{hyperref}

\usepackage{tikz}
\usetikzlibrary{matrix}

\newtheorem{theorem}{Theorem}[section]
\newtheorem{lemma}[theorem]{Lemma}
\newtheorem{proposition}[theorem]{Proposition}
\newtheorem{corollary}[theorem]{Corollary}

\theoremstyle{definition}
\newtheorem{definition}[theorem]{Definition}
\newtheorem{example}[theorem]{Example}
\newtheorem{remark}[theorem]{Remark}

\definecolor{MyGray}{rgb}{0.96,0.97,0.98}


\newcommand{\ZZ}{\mathcal{Z}}
\newcommand{\OO}{\mathcal{O}}
\newcommand{\N}{\mathbb{N}}
\newcommand{\Z}{\mathbb{Z}}
\newcommand{\Q}{\mathbb{Q}}
\newcommand{\R}{\mathbb{R}}
\newcommand{\C}{\mathbb{C}}
\newcommand{\K}{\mathbb{K}}		

\newcommand{\Cuntz}[1]{\mathcal{O}_{#1}}
\newcommand{\Tor}[1]{{Tor}(#1)}
\newcommand{\Aut}[1]{{\rm Aut}(#1)}

\newcommand{\uAut}[1]{{\rm Aut}_0(#1)}

\newcommand{\id}[1]{{\rm id}_{#1}}

\begin{document}

\title{A Dixmier-Douady Theory for strongly self-absorbing $C^*$-algebras II: the Brauer group}
\author{Marius Dadarlat}\address{MD: Department of Mathematics, Purdue University, West Lafayette, IN 47907, USA}\email{mdd@math.purdue.edu}
\author{Ulrich Pennig}\address{UP: Mathematisches Institut, Westf\"alische Wilhelms-Universit\"at M\"unster, Ein\-stein\-stra\ss e 62, 48149 M\"unster, Germany}\email{u.pennig@uni-muenster.de}	
\begin{abstract}
 We have previously shown that the isomorphism classes of orientable locally trivial fields of $C^*$-algebras
 over a compact metrizable space $X$ with fiber $D\otimes \K$, where $D$ is a strongly self-absorbing $C^*$-algebra, form an abelian group under the operation of tensor product. Moreover this group
 is isomorphic to the first group $\bar{E}^1_D(X)$ of the (reduced) generalized cohomology theory associated
to the unit spectrum of topological K-theory with coefficients in $D$.
 Here we show that all the torsion elements of the group $\bar{E}^1_D(X)$ arise from locally trivial fields with fiber $D\otimes M_n(\C)$, $n\geq 1$, for all known examples
of strongly self-absorbing $C^*$-algebras $D$. Moreover the Brauer group
generated by locally trivial fields with fiber $D\otimes M_n(\C)$, $n\geq 1$
is isomorphic to $\Tor{\bar{E}^1_D(X)}$. \\[3mm]
{\it MSC-classifier:} 46L80, 46L85, 46M20
\end{abstract}
\thanks{M.D. was partially supported by NSF grant \#DMS--1362824}
\thanks{U.P. was partially supported by the SFB 878 -- ``Groups, Geometry \& Actions''}

\maketitle

	\section{Introduction}
Let $X$ be a  compact metrizable space. Let $\K$ denote the $C^*$-algebra of compact operators
on an infinite dimensional separable Hilbert space. It is well-known that $\K\otimes \K \cong \K$ and $M_n(\C)\otimes \K\cong \K$.
Dixmier and Douady \cite{paper:DixmierDouady} showed that the isomorphism classes of locally trivial fields of $C^*$-algebras over $X$ with fiber $\K$ form
an abelian group under the operation of tensor product over $C(X)$ and this group is isomorphic to $H^3(X,\Z)$.
The torsion subgroup of $H^3(X,\Z)$ admits the following description.
Each element of $\Tor{H^3(X,\Z)}$ arises as the Dixmier-Douady class of a field $A$ which is isomorphic
to the stabilization $B\otimes \K$ of some locally trivial field of $C^*$-algebras $B$ over $X$ with all fibers isomorphic to
$M_n(\C)$ for some integer $n\geq 1$, see \cite{paper:Grothendieck}, \cite{paper:AtiyahSegal}.

In this paper we generalize this result to locally trivial fields with fiber $D\otimes \K$ where $D$ is a strongly self-absorbing $C^*$-algebra \cite{paper:TomsWinter}.
For a $C^*$-algebra $B$, we denote by  $\mathscr{C}_{B}(X)$ the isomorphism classes of locally trivial continuous fields of $C^*$-algebras over $X$ with fibers isomorphic to $B$.
The  isomorphism classes of orientable locally trivial continuous fields is denoted by $\mathscr{C}^0_{B}(X)$, see Definition~\ref{orientable}.
We have shown in \cite{paper:DadarlatP1} that $\mathscr{C}_{D\otimes \K}(X)$ is an abelian group under the operation of tensor product over $C(X)$, and moreover, this group is isomorphic to the first group $E^1_D(X)$ of a generalized cohomology theory $E_D^*(X)$ which we have proven to be isomorphic to the theory associated
to the unit spectrum of topological K-theory with coefficients in $D$, see \cite{paper:DadarlatP2}.
Similarly $(\mathscr{C}^0_{D\otimes \K}(X),\otimes)\cong \bar{E}^1_D(X)$ where $\bar{E}_D^*(X)$ is
the reduced theory associated to $E_D^*(X)$. For $D=\C$, we have, of course, $E^1_{\C}(X)\cong H^3(X,\Z)$.

We consider the stabilization map $\sigma:\mathscr{C}_{D\otimes M_n(\C)}(X)\to (\mathscr{C}_{D\otimes \K}(X),\otimes)\cong E^1_D(X)$ given by $[A]\mapsto [A\otimes \K]$ and show that its image consists entirely of torsion elements.
Moreover, if $D$ is any of the known strongly self-absorbing $C^*$-algebras, we show that the stabilization map
\[\sigma: \bigcup_{n\geq 1}\mathscr{C}_{D\otimes M_n(\C)}(X) \to \Tor{\bar{E}^1_D(X)}\]
is surjective, see Theorem~\ref{thm:applications}. In this situation $\mathscr{C}_{D\otimes M_n(\C)}(X)\cong \mathscr{C}^0_{D\otimes M_n(\C)}(X)$ by Lemma~\ref{orientable} and hence the image of the stabilization map
is contained in the reduced group $\bar{E}^1_D(X)$. 
In analogy with  the classic Brauer group generated by continuous fields of complex matrices $M_n(\C)$ \cite{paper:Grothendieck}, we introduce a Brauer group $Br_D(X)$ for locally trivial fields of C*-algebras with fibers $M_n(D)$ for $D$ a strongly self-absorbing $C^*$-algebra and establish an isomorphism $Br_D(X)\cong \Tor{\bar{E}^1_D(X)}$, see Theorem~\ref{thm:Brauer_Serre}.

Our proof is new even in the classic  case $D=\C$
whose original proof relies on an argument of Serre, see \cite[Thm.1.6]{paper:Grothendieck}, \cite[Prop.2.1]{paper:AtiyahSegal}. In the cases $D = \mathcal{Z}$ or $D = \Cuntz{\infty}$ the group $\bar{E}^1_D(X)$ is isomorphic to $H^1(X, BSU_{\otimes})$,  which appeared in \cite{paper:TelemanKofModuli}, where its equivariant counterpart played a central role.

We introduced in \cite{paper:DadarlatP1} characteristic classes
$$\delta_0: E^1_D(X)\to H^1(X, K_0(D)_{+}^{\times}) \quad \text{and}\quad  \delta_k: E^1_D(X)\to H^{2k+1}(X,\Q),\quad k\geq 1.$$
  If $X$ is connected, then $\bar{E}^1_D(X)=\mathrm{ker}(\delta_0)$.
 We show that an element $a$ belongs $\Tor{E^1_D(X)}$ if and only if $\delta_0(a)$ is a torsion element and  $\delta_k(a)=0$ for all $k\geq 1$.
 
 In the last part of the paper we show that if $A^{op}$ is the opposite C*-algebra of a locally trivial continuous field  $A$ with fiber $D\otimes \K$,
 then $\delta_{k}(A^{op})=(-1)^k\delta_{k}(A)$ for all $k\geq 0$. This shows that in general
 $A\otimes A^{op}$ is not isomorphic to a trivial field,  unlike what happens in the case $D=\C$.
 Similar arguments show that in general $[A^{op}]_{Br}\neq -[A]_{Br}$ in $Br_{D}(X)$ for $A\in \mathscr{C}_{D\otimes M_n(\C)}(X)$, see Example~\ref{remark:brauer_op}.
 
We would like to thank Ilan Hirshberg for prompting us to seek a refinement of Theorem~\ref{thm:applications} in the form of Theorem~\ref{thm:app_torsion}.

\section{Background and main result}
	
The class of strongly self-absorbing $C^*$-algebras was introduced by Toms and Winter \cite{paper:TomsWinter}.
They are separable unital $C^*$-algebras $D$ singled out by the property that there exists an isomorphism $D\to D\otimes D$
which is unitarily homotopic to the map $d\mapsto d\otimes 1_D$ \cite{paper:DadarlatKK1}, \cite{paper:WinterZStable}.

If $n\geq 2$ is a natural number we denote by $M_{n^\infty}$ the UHF-algebra $M_{n}(\C)^{\otimes \infty}.$
If $P$ is a nonempty set of primes, we denote by $M_{P^\infty}$ the UHF-algebra of infinite type
$\bigotimes_{p\in P} M_{p^\infty}.$
If $P$ is the set of all primes, then  $M_{P^\infty}$ is the universal UHF-algebra, which we denote by $M_{\Q}$.

The class  $\mathcal{D}_{pi}$ of all purely infinite strongly self-absorbing $C^*$-algebras that satisfy the Universal Coefficient Theorem
in KK-theory (UCT) was completely described in \cite{paper:TomsWinter}.
$\mathcal{D}_{pi}$ consists of the Cuntz algebras $\OO_2$, $\OO_\infty$ and of all $C^*$-algebras $M_{P^\infty}\otimes \OO_\infty$  with $P$ an arbitrary set of primes.
Let $\mathcal{D}_{qd}$ denote the class of strongly self-absorbing $C^*$-algebras which satisfy the UCT and which are
 quasidiagonal. A complete description of  $\mathcal{D}_{qd}$  has become possible
due to the recent results of Matui and Sato \cite[Cor.~6.2]{paper:Matui_Sato} that build on results of Winter~\cite{paper:WinterLocalizingElliott}, and  Lin and Niu \cite{paper:Lin_Niu}.
Thus $\mathcal{D}_{qd}$ consists of $\C$, the Jiang-Su algebra $\ZZ$ and all UHF-algebras $M_{P^\infty}$ with $P$ an arbitrary set of primes. The class $\mathcal{D}=\mathcal{D}_{qd}\cup \mathcal{D}_{pi}$ contains
all known examples of strongly self-absorbing $C^*$-algebras. It is closed under tensor products.
If $D$ is  strongly self-absorbing, then $K_0(D)$ is a unital commutative ring. The group of positive invertible elements of $K_0(D)$
is denoted by $K_0(D)_{+}^{\times}$.

Let $B$ be a  $C^*$-algebra. We denote by $\mathrm{Aut}_{0}(B)$ the path component of the identity of $\mathrm{Aut}(B)$ endowed with the point-norm topology.
Recall that we denote by $\mathscr{C}_{B}(X)$ the isomorphism classes of locally trivial continuous fields over $X$
with fibers isomorphic to $B$. The structure group of $A\in \mathscr{C}_{B}(X)$ is
$\mathrm{Aut}(B)$, and $A$ is in fact given by a principal $\mathrm{Aut}(B)$-bundle
which is  determined up to an isomorphism by an element of the homotopy classes of continuous maps from $X$ to the classifying space of the topological group $\mathrm{Aut}(B)$, denoted by $[X,B\mathrm{Aut}(B)]$.
\begin{definition}\label{orientable}
A  locally trivial continuous field $A$ of $C^*$-algebras with fiber $B$ is \emph{orientable} if its structure group can be reduced to $\mathrm{Aut}_{0}(B)$, in other words if $A$ is given by an element of $[X,B\mathrm{Aut}_{0}(B)]$.
\end{definition}
The corresponding isomorphism classes of orientable and locally trivial fields is denoted by $\mathscr{C}^0_{B}(X)$.
\begin{lemma} Let $D$ be a strongly self-absorbing $C^*$-algebra satisfying the UCT.
Then $\Aut{M_n(D)}=\uAut{M_n(D)}$ for all $n\geq 1$ and hence $\mathscr{C}_{D\otimes M_n(\C)}(X)\cong \mathscr{C}^0_{D\otimes M_n(\C)}(X)$.
\end{lemma}
\begin{proof} First we show that for any $\beta\in \Aut{D \otimes M_n(\C)}$ there exist $\alpha\in \Aut{D}$
and a unitary $u\in D \otimes M_n(\C)$ such that $\beta=u(\alpha\otimes \id{M_n(\C)})u^*$.
Let $e_{11}\in M_n(\C)$ be the rank-one projection
that appears in the canonical matrix units $(e_{ij})$ of $M_n(\C)$ and let $1_n$ be the unit of $M_n(\C)$.
Then $n[1_D\otimes e_{11}]=[1_D\otimes 1_n]$ in $K_0(D)$ and hence $n[\beta(1_D\otimes e_{11})]=n[1_D\otimes e_{11}]$
in $K_0(D)$. Under the assumptions of the lemma, it is known that $K_0(D)$ is torsion free (by \cite{paper:TomsWinter}) and that $D$ has cancellation of full projections by \cite{paper:WinterZStable} and \cite{paper:Rordam-ZStable}. It follows that there is a partial isometry $v\in D \otimes M_n(\C)$
such that $v^*v=1_D\otimes e_{11}$ and $vv^*=\beta(1_D\otimes e_{11})$.
Then $u=\sum_{i=1}^n \beta (1_D\otimes e_{i1})v (1_D\otimes e_{1i})\in D \otimes M_n(\C)$ is a unitary
such that the automorphism $u^*\beta \,u$ acts identically on $1_D\otimes M_n(\C)$.
It follows that $u^*\beta \,u=\alpha\otimes \id{M_n(\C)}$ for some $\alpha\in \Aut{D}$.
Since both $U(D \otimes M_n(\C))$ and $\Aut{D}$ are path connected by \cite{paper:TomsWinter}, \cite{paper:Rordam-ZStable} and respectively \cite{paper:DadarlatKK1} we conclude that
$\Aut{D \otimes M_n(\C)}$ is path-connected as well.
\end{proof}

Let us recall the following results contained in Cor. 3.7, Thm. 3.8 and Cor. 3.9 from \cite{paper:DadarlatP1}.
Let $D$ be a strongly self-absorbing $C^*$-algebra.
\begin{itemize}
\item[(1)] The classifying spaces $B\mathrm{Aut}(D\otimes \K)$ and $B\mathrm{Aut}_{0}(D\otimes \K)$
are infinite loop spaces giving rise to generalized cohomology theories $E^*_D(X)$ and respectively
$\bar{E}^*_D(X)$.
\item[(2)] The monoid $(\mathscr{C}_{D\otimes \K}(X),\otimes)$ is an abelian group isomorphic
to $E^1_D(X)$. Similarly, the monoid $(\mathscr{C}^0_{D\otimes \K}(X),\otimes)$ is a group isomorphic
to $\bar{E}^1_D(X)$. In both cases the tensor product is understood to be over $C(X)$.
\item[(3)]  $E_{M_{\mathbb{Q}}}^{1}(X)  \cong H^1(X,\mathbb{Q}_{+}^\times)\oplus \bigoplus_{k \geq 1} H^{2k+1}(X,\mathbb{Q})$,\\
$E_{M_{\mathbb{Q}}\otimes \mathcal{O}_\infty}^{1}(X)  \cong H^1(X,\mathbb{Q}^\times)\oplus \bigoplus_{k \geq 1} H^{2k+1}(X,\mathbb{Q})$,
\item[(4)]  $\bar{E}_{M_{\mathbb{Q}}\otimes \mathcal{O}_\infty}^{1}(X) \cong \bar{E}_{M_{\mathbb{Q}}\otimes \mathcal{O}_\infty}^{1}(X) \cong  \bigoplus_{k \geq 1} H^{2k+1}(X,\mathbb{Q}).$
\item[(5)] If $D$ satisfies the UCT  then
$D\otimes M_{\mathbb{Q}}\otimes \mathcal{O}_\infty \cong M_{\mathbb{Q}}\otimes \mathcal{O}_\infty,$
by \cite{paper:TomsWinter}.
Therefore the tensor product operation $A\mapsto A \otimes M_{\mathbb{Q}}\otimes \mathcal{O}_\infty$ induces
maps $$\mathscr{C}_{D\otimes \K}(X)\to \mathscr{C}_{M_{\mathbb{Q}}\otimes \mathcal{O}_\infty \otimes \K}(X), \quad \mathscr{C}^0_{D\otimes \K}(X)\to \mathscr{C}^0_{M_{\mathbb{Q}}\otimes \mathcal{O}_\infty \otimes \K}(X)\quad \text{and hence maps}$$
$$ E^1_D(X)\stackrel{\delta}{\longrightarrow}  E^1_{M_{\mathbb{Q}}\otimes \mathcal{O}_\infty}(X)\cong H^1(X,\mathbb{Q}^\times)\oplus \bigoplus_{k \geq 1} H^{2k+1}(X,\mathbb{Q}),$$
$$\delta(A)=(\delta^s_0(A),\delta_1(A),\delta_2(A),\dots ), \quad \delta_k(A)\in  H^{2k+1}(X,\mathbb{Q}),$$

$$\bar{E}^1_D(X)\stackrel{\bar\delta}{\longrightarrow}\bar{E}^1_{M_{\mathbb{Q}}\otimes \mathcal{O}_\infty}(X)\cong  \bigoplus_{k \geq 1} H^{2k+1}(X,\mathbb{Q}),$$
$$\bar{\delta}(A)=(\delta_1(A),\delta_2(A),\dots ), \quad \delta_k(A)\in  H^{2k+1}(X,\mathbb{Q}).$$
\end{itemize}
The invariants $\delta_k(A)$  are called the rational characteristic classes of the continuous field $A$, see  \cite[Def.4.6]{paper:DadarlatP1}.
The first class $\delta^s_0$ lifts to a map $\delta_0: E^1_D(X)\to H^1(X,K_0(D)_+^\times)$  induced by the morphism of groups $\Aut{D\otimes \K}\to \pi_0(\Aut{D\otimes \K})\cong K_0(D)_+^\times$. $\delta_0(A)$  represents the obstruction to reducing the structure group of $A$ to $\uAut{D\otimes \K}$.

\begin{proposition}  A continuous field $A\in \mathscr{C}_{D\otimes \K}(X)$ is orientable if and only if $\delta_0(A)=0$. If $X$ is connected, then $\bar{E}^1_D(X)\cong \mathrm{ker}(\delta_0)$.
\end{proposition}
\begin{proof}
Let us recall from \cite[Cor.~2.19]{paper:DadarlatP1} that
 there is an exact sequence of
topological groups
\begin{equation}\label{eqn:basic} 1\to \mathrm{Aut}_{0}(D\otimes \K)\to \mathrm{Aut}(D\otimes \K)\stackrel{\pi}{\longrightarrow} K_0(D)^{\times}_{+}\to 1.\end{equation}
The map $\pi$ takes an automorphism $\alpha$ to $[\alpha(1_D\otimes e)]$ where $e\in \K$ is a rank-one projection.
If $G$ is a topological group and $H$ is a normal subgroup of $G$ such that
$H\to G \to G/H$  is a principal $H$-bundle,    then there is a homotopy fibre sequence
$G/H\to  BH \to BG \to B(G/H)$ and hence an exact sequence of pointed sets
$[X,G/H]\to [X,BH] \to [X,BG] \to [X,B(G/H)]$.
In particular, in the case of the fibration \eqref{eqn:basic} we obtain
\begin{equation}\label{eqn:basic+}
[X,K_0(D)^{\times}_{+}]\to [X,B\mathrm{Aut}_{0}(D\otimes \K)] \to [X,B\mathrm{Aut}(D\otimes \K)] \stackrel{\delta_0}{\longrightarrow} H^1(X, K_0(D)^{\times}_{+}).\end{equation}
   A continuous field $A\in \mathscr{C}^0_{D\otimes \K}(X)$ is associated to a principal $\mathrm{Aut}(D\otimes \K)$-bundle  whose classifying map gives a unique element in $[X,B\mathrm{Aut}(D\otimes \K)]$ whose image
in $H^1(X, K_0(D)^{\times}_{+})$ is denoted by $\delta_0(A)$.
It is clear from \eqref{eqn:basic+} that the class $\delta_0(A)\in H^1(X, K_0(D)^{\times}_{+})$ represents the obstruction for reducing this bundle
to a principal $\mathrm{Aut}_{0}(D\otimes \K)$-bundle.
If $X$ is connected, $[X,K_0(D)^{\times}_{+}]=\{*\}$ and hence $\bar{E}^1_D(X)\cong \mathrm{ker}(\delta_0)$.
\end{proof}
\begin{remark}
 If $D=\C$ or $D=\ZZ$ then $A$ is automatically orientable since in those cases $K_0(D)^{\times}_{+}$ is the trivial group.\end{remark}

\begin{remark} \label{rem:delta_suspension}
Let $Y$ be a compact metrizable space and let $X = \Sigma Y$ be the suspension of $Y$. Since the rational K\"unneth isomorphism and the Chern character on $K^0(X)$ are compatible with the ring structure on $K_0(C(Y) \otimes D)$, we obtain a ring homomorphism
\[
	{\rm ch} \colon K_0(C(Y) \otimes D) \to K^0(Y) \otimes K_0(D) \otimes \Q \to \prod_{k = 0}^{\infty} H^{2k}(Y, \Q) =: H^{\rm ev}(Y,\Q)\ ,
\]
which restricts to a group homomorphism ${\rm ch} \colon \bar{E}_D^0(Y) \to SL_1(H^{\rm ev}(Y,\Q))$, where the right hand side denotes the units, which project to $1 \in H^0(Y, \Q)$. If $A$ is an orientable locally trivial continuous field with fiber $D \otimes \K$ over $X$, then we have
\begin{equation} \label{eqn:delta_susp}
	\delta_k(A) = \log {\rm ch}(f_A) \in H^{2k}(Y, \Q) \cong H^{2k+1}(X, \Q)\ ,
\end{equation}
where $f_A \colon Y \to \Omega B\uAut{D \otimes \K} \simeq \uAut{D \otimes \K}$ is induced by the transition map of $A$. The homomorphism ${\rm log} \colon SL_1(H^{\rm ev}(Y, \Q)) \to H^{\rm ev}(Y, \Q)$ is the rational logarithm from \cite[Section 2.5]{paper:Rezk}. For the proof of (\ref{eqn:delta_susp}) it suffices to treat the case $D = M_{\Q} \otimes \Cuntz{\infty}$, where it can be easily checked on the level of homotopy groups, but since $\bar{E}^0_D(Y)$ and $H^{\rm ev}(Y,\Q)$ have rational vector spaces as coefficients this is enough.
\end{remark}

{
\begin{lemma}\label{lemma:corners}
Let $D$ be a strongly self-absorbing $C^*$-algebra in the class $\mathcal{D}$. If $p\in D\otimes \K$ is a projection such that $[p]\neq 0$ in $K_0(D)$, then there is an integer $n\geq 1$ such that $[p] \in nK_0(D)^{\times}_{+}$. If $[p] \in nK_0(D)^{\times}_{+}$, then $p(D\otimes \K)p\cong M_{n}(D)$. Moreover, if $n,m \geq 1$, then $M_n(D)\cong M_m(D)$ if and only if $n K_0(D)_{+}^{\times}=mK_0(D)_{+}^{\times}$.
\end{lemma}
}
\begin{proof} Recall that $K_0(D)$ is an ordered unital ring with unit $[1_D]$ and with positive elements $K_0(D)_{+}$ corresponding to classes of projections in $D\otimes \K$. The group of invertible elements is denoted by $K_0(D)^{\times}$
and $K_0(D)^{\times}_{+}$ consists of classes $[p]$ of projections
 $p\in D\otimes \K$ such that $[p]\in K_0(D)^{\times}$. It was shown in \cite[Lemma~2.14]{paper:DadarlatP1}
that if $p\in D\otimes \K$ is a projection, then $[p]\in K_0(D)^{\times}_{+}$ if and only if $p(D\otimes \K)p\cong D$. The ring $K_0(D)$ and the group $K_0(D)^{\times}_{+}$ are known for all $D\in \mathcal{D}$, \cite{paper:TomsWinter}.
In fact $K_0(D)$ is a unital subring of $\mathbb{Q}$, $K_0(D)_{+}= \Q_+\cap K_0(D)$ if $D\in \mathcal{D}_{qd}$
and $K_0(D)_{+}=K_0(D)$ if $D\in \mathcal{D}_{pi}.$
Moreover:
\begin{itemize}
\item[] $K_0(\C)\cong K_0(\ZZ)\cong K_0(\mathcal{O}_\infty)\cong \Z$, $K_0(\mathcal{O}_2)=\{0\}$,

\item[]  $K_0(M_{P^\infty})\cong K_0(M_{P^\infty}\otimes \OO_\infty) \cong \Z[1/P]\cong\bigotimes\limits_{p\in P} \Z[1/p]\cong \{np_1^{k_1}p_2^{k_2}\cdots p_r^{k_r}\colon p_i\in P, n, k_i\in \Z\},$

\item[]   $K_0(\C)_{+}^{\times}\cong K_0(\ZZ)_{+}^{\times}=\{1\}$, $K_0(\mathcal{O}_\infty)_{+}^{\times}=\{\pm 1\}$,

\item[]  $K_0(M_{P^{\infty}})_{+}^{\times}\cong \{p_1^{k_1}p_2^{k_2}\cdots p_r^{k_r}\colon p_i\in P, k_i\in \Z\}$.
\item[]  $K_0(M_{P^{\infty}}\otimes \OO_\infty)_{+}^{\times}\cong \{\pm p_1^{k_1}p_2^{k_2}\cdots p_r^{k_r}\colon p_i\in P, k_i\in \Z\}$.
\end{itemize}
In particular, we see  that in all cases $K_0(D)_{+}=\N\cdot K_0(D)^{\times}_{+}$, 
{which proves the first statement. If $p\in D\otimes \K$ is a projection such that $[p] \in nK_0(D)^{\times}_{+}$, then there is a projection $q \in  D\otimes \K$ such that $[q]\in K_0(D)^{\times}_{+}$ and $[p]=n[q]=[\mathrm{diag}(q,q,\dots,q)]$. Since $D$ has cancellation of full projections, it follows then immediately that $p(D\otimes \K)p\cong M_n(D)$ proving the second part.}

To show the last part of the lemma, suppose now that $\alpha:D\otimes M_n(\C)\to D\otimes M_m(\C)$ is a $*$-isomorphism.
Let $e\in M_n(\C)$ be a rank one projection.
Then $\alpha(1_D\otimes e) (D\otimes M_m(\C))\alpha(1_D\otimes e) \cong D$.
By \cite[Lemma~2.14]{paper:DadarlatP1} it follows that $\alpha_*[1_D]=[\alpha(1_D\otimes e)]\in K_0(D)_{+}^{\times}$. Since $\alpha$ is unital, $\alpha_*(n[1_D])=m[1_D]$ and hence $m[1_D]\in n K_0(D)_{+}^{\times}$. This is equivalent to $n K_0(D)_{+}^{\times}=mK_0(D)_{+}^{\times}$.

Conversely, suppose that $m[1_D]=n u$ for some $u\in  K_0(D)_{+}^{\times}$.
Let $\alpha\in \Aut{D\otimes \K}$ be such that $[\alpha(1_D\otimes e)]=u$.
Then $\alpha_*(n[1_D])=nu=m[1_D]$. This implies
that $\alpha$ maps a corner of $D\otimes\K$ that is isomorphic to $M_n(D)$ to a corner that is isomorphic
to $M_m(D)$.
\end{proof}

{
\begin{corollary} \label{cor:preimage_units}
Let $D\in \mathcal{D}$ and let $\theta \colon D\otimes M_{n^r}(\C) \to D \otimes M_{n^{\infty}}$ be a unital inclusion induced by some unital embedding $M_{n^r}(\C)\to M_{n^{\infty}}$, where  $n\geq 2, r\geq 0$.
Let $R$ be the set of prime factors of $n$.
 Then, under the canonical isomorphism $ K_0(D\otimes M_{n^r}(\C)) \cong K_0(D)$, we have
\[
	\theta_*^{-1}(K_0(D \otimes M_{n^{\infty}})^{\times}_+) = \bigcup_r rK_0(D)^{\times}_+ \subset K_0(D)
\] where $r$ runs through the set of all products of the form $\prod_{q\in R}q^{k_q}$, $k_q\in \N \cup\{0\}$.
\end{corollary}
}
{
\begin{proof} \label{pf:preimage_units}
From Lemma~\ref{lemma:corners} we see that $K_0(D) \cong \Z[1/P]$ for a (possibly empty) set of primes $P$.  The order structure is the one induced by $(\Q, \Q_+)$ if $D$ is  quasidiagonal or $K_0(D)^{+} = \Z[1/P]$ if $D$ is purely infinite. If $R \subseteq P$, then $\theta$ induces an isomorphism on $K_0$ and the statement is true, since $\theta_*$ is order preserving and $\Z[1/R]^{\times} \subseteq K_0(D)^{\times}$. Thus, we may assume that $ R\nsubseteq P$. Let $S= P \cup R$ and thus $K_0(D \otimes M_{n^{\infty}})\cong \Z[1/S]$. The map $\theta_*$ induces the canonical  inclusion $\Z[1/P] \hookrightarrow \Z[1/S]$. We can write $x \in \Z[1/P]$ as 
\[
	x = m \cdot \prod_{p \in P} p^{r_p} \cdot \prod_{q\in R\setminus P}q^{k_q} 
\]
with $m \in \Z$ relatively prime to all $p \in P$ and $q \in R$, only finitely many $r_p \in \Z$ non-zero and 
$k_q\in \N\cup \{0\}$. From this decomposition we see that $x$ is invertible in $\Z[1/S]$ if and only if $m=\pm 1$. This concludes the proof since $p^{r_p}\in K_0(D)^{\times}_{+}$.
\end{proof}
}
{
\begin{remark}
Let $q \in D \otimes \K$ be a projection and let $\alpha \in \Aut{D \otimes \K}$. As in \cite[Lemma 2.14]{paper:DadarlatP1} we have that $[\alpha(q)] = [\alpha(1 \otimes e)] \cdot [q]$ with $[\alpha(1 \otimes e)] \in K_0(D)^{\times}_{+}$. Thus, the condition $[q] \in nK_0(D)^{\times}_{+}$ for $n \in \N$ is invariant under the action of $\Aut{D \otimes \K}$ on $K_0(D)$. Given $A \in \mathscr{C}_{D \otimes \K}(X)$, a projection $p \in A$, $x_0 \in X$ and an isomorphism $\phi \colon A(x_0) \to D \otimes \K$ the condition $[\phi(p(x_0))] \in nK_0(D)^{\times}_{+}$ is independent of $\phi$. Abusing the notation we will write this as $[p(x_0)] \in nK_0(D)^{\times}_{+}$. 
\end{remark}
}

\begin{corollary}\label{cor:global_n}  Let $D\in \mathcal{D}$ and let $A\in \mathscr{C}_{D\otimes \K}(X)$ with $X$ a connected compact metrizable space. {If $p \in A$ is a projection such that $[p(x_0)]\in nK_0(D)^{\times}_{+}$ for some point $x_0$, then $(p A p)(x) \cong M_n(D)$ for all $x\in X$ and hence $pAp\in \mathscr{C}_{D\otimes M_n(\C)}(X)$. If $p \in A$ is a projection with $[p(x_0)] \in K_0(D) \setminus \{0\}$, then $[p(x_0)] \in nK_0(D)^{\times}_{+}$ for some $n \in \N$.}
\end{corollary}

\begin{proof} Let $V_1,..,V_k$ be a finite cover of $X$ by compact sets such that there are bundle isomorphisms
$\phi_i:A(V_i)\cong C(V_i)\otimes D \otimes \K$. Let $p_i$ be the image of  the restriction of $p$ to
$V_i$ under $\phi_i$. After refining the cover $(V_i)$, if necessary, we may
assume that $\|p_i(x)-p_i(y)\|<1$  for all $x,y \in V_i$. This allows us to find a unitary $u_i$ in the multiplier
algebra of $C(V_i)\otimes D \otimes \K$ such that after replacing $\phi_i$ by $u_i\phi_i u_i^*$
and $p_i$ by $u_ip_i u_i^*$, we may assume that $p_i$ are constant projections.
{Since $X$ is connected and $[p(x_0)] \in nK_0(D)^{\times}_{+}$ by assumption, it follows from $[p_i(x_0)] \in nK_0(D)^{\times}_{+}$ for $x_0 \in V_i$ and the above remark that $[p_j(x)] \in nK_0(D)^{\times}_{+}$ for all $1 \leq j \leq k$ and all $x \in V_j$. Then Lemma~\ref{lemma:corners} implies $(pAp)(V_j)\cong C(V_j)\otimes M_{n}(D)$. By Lemma~\ref{lemma:corners} we also have that $[p(x_0)] \neq 0$ implies $[p(x_0)] \in nK_0(D)^{\times}_{+}$ for some $n \in \N$ proving the statement about the case $[p(x_0)] \in K_0(D) \setminus \{0\}$.}
\end{proof}

We study the image of the stabilization map
	 $$\mathscr{C}_{D\otimes M_n(\C)}(X)\to \mathscr{C}_{D\otimes \K}(X)$$
	 induced by the map $A\mapsto A\otimes \K$, or equivalently by the map
	 $$\Aut{D\otimes M_n(\C)}\to \Aut{D\otimes M_n(\C) \otimes \K} \cong \Aut{D\otimes \K}.$$
	
Let us recall that  $\mathcal{D}$ denotes the class of strongly self-absorbing $C^*$-algebras which satisfy the UCT and which are
either quasidiagonal or purely infinite.
\begin{theorem}\label{thm:applications} Let $D$ be a strongly self-absorbing $C^*$-algebra in the class $\mathcal{D}$.  Let $A$ be a locally trivial continuous field of $C^*$-algebras over a connected compact metrizable space $X$
such that  $A(x) \cong D \otimes \K$ for all $x\in X$.
The following assertions are equivalent:
\begin{itemize}
\item[(1)] $\delta_k(A)=0$ for all $k\geq 0$.
\item[(2)] The field $A\otimes M_{\mathbb{Q}}$ is trivial.

\item[(3)] There is  an integer $n\geq 1$ and a  unital locally trivial continuous field $\mathcal{B}$ over $X$ with all fibers isomorphic to $M_n(D)$ such that $A\cong \mathcal{B}\otimes \K$.	
\item[(4)] $A$ is orientable and $A^{\otimes m} \cong C(X)\otimes D \otimes \K$  for some $m \in \N$.\end{itemize}
\end{theorem}
\begin{proof} The statement is immediately verified if $D\cong \OO_2$. Indeed all locally trivial fields with fiber
$\OO_2\otimes \K$ are trivial since $\Aut{\OO_2\otimes \K}$ is contractible by \cite[Cor.~17 \& Thm.~2.17]{paper:DadarlatP1}. For the remainder of the proof we may therefore assume that $D\ncong \OO_2$.

(1) $\Leftrightarrow$ (2) If $D\in \mathcal{D}_{qd}$, then it is known that $D\otimes M_{\Q}\cong M_\Q$.
Similarly, if $D\in \mathcal{D}_{pi}$ and $D\ncong \OO_2$ then  $D\otimes M_\Q\cong \OO_{\infty}\otimes M_{\Q}$.
If $A$ is as in the statement, then $A\otimes M_{\Q}$ is a locally trivial field whose  fibers are all isomorphic to either
$M_\Q \otimes \K$ or  to $\OO_{\infty}\otimes M_{\Q}\otimes \K$. In either case, it was shown in
\cite[Cor.~4.5]{paper:DadarlatP1} that such a field is trivial if and only if  $\delta_k(A)=0$ for all $k\geq 0$. As reviewed earlier in this section, this follows from the explicit computation of  $E_{M_{\mathbb{Q}}}^{1}(X)$ and $E_{M_{\mathbb{Q}}\otimes \mathcal{O}_\infty}^{1}(X)$.

(2) $\Rightarrow$ (3)
 Assume now that $A\otimes M_{\Q}$ is trivial, i.e. $A\otimes M_{\Q}\cong C(X)\otimes  D\otimes M_{\mathbb{Q}} \otimes \K .$ Let $p\in A\otimes M_{\Q}$ be the projection that corresponds under this isomorphism to the projection $ 1\otimes e \in C(X)\otimes D \otimes  M_{\mathbb{Q}} \otimes \K $ where $1$ is the unit of the $C^*$-algebra
$C(X)\otimes D \otimes M_{\mathbb{Q}}$ and $e\in \K$ is a rank-one projection. Then $[p(x)]\neq 0$ in
$K_0(A(x)\otimes M_{\Q})$ for all $x\in X$ (recall that $D\ncong \OO_2$).
Let us write $M_{\mathbb{Q}}$ as the direct limit of an increasing sequence of its subalgebras $M_{k(i)}(\C).$
Then $A\otimes M_{\mathbb{Q}}$ is the direct limit of the sequence
$A_i=A\otimes M_{k(i)}(\C) $. It follows that there exist  $i\geq 1$ and a projection $p_i\in A_i$ such that
$\|p-p_i\|<1.$ Then $\|p(x)-p_i(x)\|<1$  and so $[p_i(x)]\neq 0$ in $K_0(A_i(x))$ for each $x\in X$,
since its image in $K_0(A(x)\otimes M_{\Q})$ is equal to $[p(x)]\neq 0$.
 Let us consider the locally trivial unital field
 $\mathcal{B}:=p_i(A\otimes M_{k(i)}(\C) )p_i$.
 Since the fibers of $A\otimes M_{k(i)}(\C)$ are isomorphic to $D \otimes \K \otimes M_{k(i)}(\C) \cong D \otimes \K$, it follows by Corollary~\ref{cor:global_n} that there is $n\geq 1$ such that all fibers of
  $\mathcal{B}$ are isomorphic to $M_n(D)$.
  Since $\mathcal{B}$ is isomorphic to a full corner of $A\otimes \K$, it follows  by \cite{paperL.G.Brown.Stable.Isom} that
  $A \otimes \K\cong \mathcal{B} \otimes \K$. We conclude by noting that since $A$ is locally trivial and each fiber is stable, then $A\cong A\otimes \K$
by \cite{paper.Hirshberg.Rordam.Winter} and so $A \cong \mathcal{B} \otimes \K$.

(3) $\Rightarrow$ (2) This implication holds for any strongly self-absorbing $C^*$-algebra $D$.
Let $A$ and $\mathcal{B}$ be as in (3).
Let us note that $\mathcal{B}\otimes M_{\Q}$
 is a unital locally trivial field with all fibers isomorphic to the strongly self-absorbing $C^*$-algebra  $D\otimes M_{\mathbb{Q}}$.
 Since $\Aut{D\otimes M_{\mathbb{Q}}}$ is contractible by \cite[Thm. 2.3]{paper:DadarlatP1}, it follows that $\mathcal{B}\otimes M_\Q$ is trivial. We conclude that
$A \otimes M_{\mathbb{Q}}\cong (\mathcal{B} \otimes M_{\mathbb{Q}} )\otimes \K\cong C(X)  \otimes D\otimes M_{\mathbb{Q}}\otimes \K.$

(2) $\Leftrightarrow$ (4) This equivalence holds for any strongly self-absorbing $C^*$-algebra $D$ if $A$ is orientable. In particular we do not need to assume that $D$ satisfies the UCT.
In the UCT case we note that since the map $K_0(D)\to K_0(D\otimes M_\Q)$ is injective,
it follows that $A$ is orientable if and only if $A\otimes M_\Q$ is orientable, i.e. $\delta_0(A)=0$ if and only if $\delta^s_0(A)=0$.
Since $\delta_0(A)=0$, $A$ is determined up to isomorphism by its class $[A]\in \bar{E}^1_{D}(X)$.
To complete the proof it suffices to show that the kernel of the map $\tau: \bar{E}^1_{D}(X)\to
\bar{E}^1_{D\otimes M_\Q }(X)$, $\tau[A]= [A\otimes M_{\Q}]$,  consists entirely of torsion elements. Consider the natural
transformation of cohomology theories: $$\tau\otimes \mathrm{id}_{\Q}: \bar{E}^*_{D}(X)\otimes \Q
\to \bar{E}^*_{D\otimes M_\Q}(X)\otimes \Q \cong \bar{E}^*_{D\otimes M_\Q}(X).$$
If $D\neq \C$, it induces an isomorphism on coefficients since $\bar{E}^{-i}_{D}(pt)=\pi_i (\uAut{D\otimes \K})\cong K_i(D)$ by \cite[Thm.2.18]{paper:DadarlatP1} and since the map $K_i(D)\otimes \Q \to K_i(D\otimes M_\Q )$ is bijective. We conclude that the kernel of $\tau$ is a torsion group.
The  same property holds for $D=\C$ since $\bar{E}^*_{\C}(X)$  is a direct summand of $\bar{E}^*_{\ZZ}(X)$
by \cite[Cor.3.8]{paper:DadarlatP1}.
 \end{proof}	
 {
\begin{theorem} \label{thm:app_torsion}
Let $D$, $X$ and $A$ be as in Theorem~\ref{thm:applications} and let
 $n\geq 2$ be an integer. 
  The following assertions are equivalent:
\begin{itemize}
\item[(1)] The field $A\otimes M_{n^{\infty}}$ is trivial.
\item[(2)] There is a $k \in \N$ and a unital locally trivial continuous field $\mathcal{B}$ over $X$ with all fibers isomorphic to $M_{n^k}(D)$ such that $A\cong \mathcal{B}\otimes \K$.	
\item[(3)] $A$ is orientable and $A^{\otimes n^k} \cong C(X)\otimes D \otimes \K$  for some $k \in \N$.
\end{itemize}
\end{theorem}
}

\begin{proof} \label{pf:app_torsion}
By reasoning as in the proof of Theorem~\ref{thm:applications}, we may assume that $D\ncong \Cuntz{2}$.

(1) $\Rightarrow$ (2): By assumption the continuous field $A \otimes M_{n^{\infty}}$ is trivializable and hence it satisfies the global Fell condition of \cite{paper:DadarlatP1}. This means that there is a full projection $p_{\infty} \in A \otimes M_{n^{\infty}}$ with the property that $p_{\infty}(x) \in K_0(A(x) \otimes M_{n^{\infty}})_+^{\times}$ for all $x \in X$. Let $\nu_i \colon M_{n^i}(\C) \to M_{n^{\infty}}$ be a unital inclusion map. Since $A \otimes M_{n^{\infty}}$ is the inductive limit of the sequence 
\[
	A \to A \otimes M_n(\C) \to \dots \to A \otimes M_{n^i}(\C) \to A \otimes M_{n^{i+1}}(\C) \to \dots 
\] 
there is an $i \in \N$ and a full projection $p \in A \otimes M_{n^i}(\C)$ with $\|(\id{A} \otimes \nu_i)(p) - p_{\infty} \| < 1$. Fix a point $x_0 \in X$. 
 Let $\theta \colon A(x_0)\otimes M_{n^i}(\C) \to A(x_0) \otimes M_{n^{\infty}}$ be the unital inclusion induced by $\nu_i$.
 Note that $\theta_*([p(x_0)]) =(\id{A(x_0)} \otimes \nu_i)_*([p(x_0)]) = [p_{\infty}(x_0)] \in K_0(A(x_0) \otimes M_{n^{\infty}})_+^{\times}$.  By Corollary~\ref{cor:preimage_units} this implies that $[p(x_0)] \in rK_0(A(x_0))^{\times}_+$ for some $r\in \N$  that  divides $n^k$ for some $k \in \N\cup \{0\}$. 
 Then 
 $\mathcal{B}_0:=p(A \otimes M_{n^i}(\C))p \in \mathscr{C}_{D \otimes M_{r}(\C)}(X)$ by Corollary~\ref{cor:global_n}.
  Write $n^k=mr$ with $m\in \N$. It follows that $\mathcal{B}:=\mathcal{B}_0\otimes M_m(\C)\in \mathscr{C}_{D \otimes M_{n^k}(\C)}(X)$.
The fact that $\mathcal{B} \otimes \K \cong A$ follows just as in step (2) $\Rightarrow$ (3) in the proof of Theorem~\ref{thm:applications}.

{
(2) $\Rightarrow$ (1): This is just the same argument as step (3) $\Rightarrow$ (2) in the proof of Theorem~\ref{thm:applications}.
}

{
(1) $\Leftrightarrow$ (3): }
{The orientability of $A$ follows from Theorem~\ref{thm:applications}.
}
{Observe that the elements $[A] \in \mathscr{C}^0_{D \otimes \K}(X) = \bar{E}^1_D(X)$ such that $n^k[A]=0$ or equivalently $A^{\otimes n^k}$ is trivializable for some $k \in \N\cup\{0\}$ coincide precisely with the elements in the kernel of the group homomorphism $\bar{E}^1_D(X) \to \bar{E}^1_D(X) \otimes \Z[\tfrac{1}{n}]$. Since $\Z[\tfrac{1}{n}]$ is flat, it follows that $X \mapsto \bar{E}^*_D(X) \otimes \Z[\tfrac{1}{n}]$ still satisfies all axioms of a generalized cohomology theory. In particular, we have the following commutative diagram of natural transformations of cohomology theories:
\[
	\xymatrix{
		\bar{E}^*_D(X) \ar[r] \ar[d] & \bar{E}^*_{D \otimes M_{n^{\infty}}}(X) \ar[d]^-{\cong} \\
		\bar{E}^*_D(X) \otimes \Z[\tfrac{1}{n}] \ar[r] & \bar{E}^*_{D \otimes M_{n^{\infty}}}(X) \otimes \Z[\frac{1}{n}]
	}
\]
where the isomorphism on the right hand side can be checked on the coefficients. A similar argument shows that for $D \neq \C$ the bottom homomorphism is an isomorphism. Thus the kernel of the left vertical map agrees with the one of the upper horizontal map in this case. For $D = \C$ we can use that $\bar{E}^*_{\C}(X)$ embeds as a direct summand into $\bar{E}^*_{\ZZ}(X)$ via the natural $*$-homomorphism $\C \to \ZZ$ \cite[Cor.~4.8]{paper:DadarlatP1}. In particular, $\bar{E}^*_{\C}(X) \otimes \Z[\tfrac{1}{n}] \to \bar{E}^*_{\ZZ}(X) \otimes \Z[\tfrac{1}{n}]$ is injective.
}
\end{proof} 

\begin{corollary}\label{cor: n-torsion}
Let $D$ and $X$ be as in Theorem~\ref{thm:applications}. Then any element $x\in \bar{E}_D^1(X)$ with $nx=0$ is represented by the stabilization of a unital locally trivial field over $X$ with all fibers isomorphic to $M_{n^k}(D)$ for some $k\geq 1$. Moreover if $A\in \mathscr{C}_{D\otimes \K}(X)$, then $A\otimes M_{\mathbb{Q}}$ is trivial  
$\Leftrightarrow$ $A\otimes M_{n^\infty}$ is trivial for some $n\in \N$ $\Leftrightarrow$ $A$ is orientable and $n^k[A]=0$ in $\bar{E}_D^1(X)$ for  some $k\in \N$ and some $n \in \N$.
\end{corollary}
(An example from \cite{paper:AtiyahSegal} for $D=\C$ shows that in general one cannot always arrange that $k=1$.)
\begin{proof} The first part follows from Theorem~\ref{thm:app_torsion}. Indeed, condition (3) of that theorem is equivalent to requiring that $A$ is orientable and $n^k[A]=0$ in $\bar{E}^1_D(X)$. The second part follows from Theorems~\ref{thm:applications} and \ref{thm:app_torsion}.\end{proof}

\begin{definition} \label{def:Brauer} Let $D$ be  a strongly self-absorbing $C^*$-algebra. If $X$ is connected
compact metrizable space we define the Brauer group $Br_D(X)$ as
equivalence classes of continuous fields $A\in \bigcup_{n\geq 1}\mathscr{C}_{M_n(D)}(X)$.
Two continuous fields $A_i\in \mathscr{C}_{M_{n_i}(D)}(X)$, $i=1,2$ are equivalent, if
$$A_1 \otimes p_1C(X, M_{N_1}(D))p_1\cong A_2 \otimes p_2C(X, M_{N_2}(D))p_2 ,$$
for some full projections $p_i\in C(X,M_{N_i}(D)).$ We denote by $[A]_{Br}$ the class of $A$ in $Br_D(X)$.
The multiplication on $Br_D(X)$ is induced by the tensor product operation, after fixing an isomorphism
$D\otimes D\cong D$.  We will show in a moment that the monoid $Br_D(X)$ is a group.
\end{definition}

\begin{remark} 
It is worth noting the following two alternative descriptions of the Brauer group.
(a) If $D\in \mathcal{D}$ is quasidiagonal, then 
two  continuous fields $A_i\in \mathscr{C}_{M_{n_i}(D)}(X)$, $i=1,2$ have equal classes in $Br_{D}(X)$, if and only if
$A_1 \otimes p_1C(X, M_{N_1}(\C))p_1\cong A_2 \otimes p_2C(X, M_{N_2}(\C))p_2 ,$
for some full projections $p_i\in C(X,M_{N_i}(\C)).$ 
(b) If $D\in \mathcal{D}$ is purely infinite, then 
two  continuous fields $A_i\in \mathscr{C}_{M_{n_i}(D)}(X)$, $i=1,2$ have equal classes in $Br_{D}(X)$, if and only if
$A_1 \otimes p_1C(X, M_{N_1}(\OO_\infty))p_1\cong A_2 \otimes p_2C(X, M_{N_2}(\OO_\infty))p_2 ,$
for some full projections $p_i\in C(X,M_{N_i}(\OO_\infty)).$ 
In order to justify (a) we observe that if $D$ is quasidiagonal,  then every projection $p\in C(X, M_{N}(D))$   has a multiple $p(m):=p \otimes 1_{M_{m}}(\C)$ such that $p(m)$ 
is Murray-Von Neumann  equivalent to a projection in $C(X, M_{Nm}(\C))\otimes 1_{D}\subset C(X, M_{Nm}(\C))\otimes D$ and that $A_i \otimes D\cong A_i$ by \cite{paper.Hirshberg.Rordam.Winter}. For (b) we note that if $D$ is purely infinite, then then every projection $p\in C(X, M_{N}(D))$   has a multiple $p \otimes 1_{M_{m}}(\C)$ that
is Murray-Von Neumann  equivalent to a projection in $C(X, M_{Nm}(\OO_\infty))\otimes 1_{D}.$
\end{remark}

One has the following generalization of a result of Serre, \cite[Thm.1.6]{paper:Grothendieck}.
\begin{theorem}\label{thm:Brauer_Serre}  Let $D$ be a strongly self-absorbing $C^*$-algebra in $\mathcal{D}$. 
\begin{itemize}
 \item[(i)] $\Tor{\bar{E}^1_D(X)}= ker \left(\bar{E}^1_D(X) \stackrel{\bar{\delta}}\longrightarrow  \bigoplus_{k \geq 1} H^{2k+1}(X,\mathbb{Q})\right)$
 \item[(ii)] The map $\theta: Br_D(X)\to \Tor{\bar{E}^1_D(X)}$, $[A]_{Br}\mapsto [A\otimes \K]$ is an isomorphism of groups.
 \end{itemize}
\end{theorem}
\begin{proof}  (i) was established in the last part of the proof of 
Theorem~\ref{thm:applications}.

(ii) We denote by $L_p$ the continuous field $p\,C(X, M_{N}(D))p$.
Since $L_p \otimes \K \cong C(X, D\otimes \K)$ it follows that the map $\theta$ is a well-defined morphism of monoids.

We use the following observation. Let $\theta:S \to G$ be a unital surjective morphism
of commutative monoids with units denoted by 1. Suppose that $G$ is a group and  that
$\{s\in S\colon \theta(s)=1\}=\{1\}$. Then $S$ is a group and $\theta$ is an isomorphism.
Indeed if $s\in S$, there is $t\in S$ such that $\theta(t)=\theta(s)^{-1}$ by surjectivity of $\theta$.
Then $\theta(st)=\theta(s)\theta(t)=1$ and so $st=1$. It follows that $S$ is a group and that $\theta$ is injective.
 
 We are going to apply this observation to the map $\theta: Br_D(X)\to \Tor{\bar{E}^1_D(X)}$.  By condition (3) of Theorem~\ref{thm:applications} we see that $\theta$ is surjective. Let us  determine the set $\theta^{-1}(\{0\})$. We are going to show that
  if $B\in \mathscr{C}_{D\otimes M_n(\C)}(X)$, then $[B\otimes \K]=0$ in $\bar{E}^1_D(X)$ if and only if
\[B\cong p \left(C(X)\otimes D \otimes M_N(\C)\right)p\cong \mathcal{L}_{C(X, D)} (p \,C(X, D)^N)\] for some selfadjoint projection $p\in C(X)\otimes D \otimes M_N(\C) \cong M_N(C(X,D))$.
  Let $B\in \mathscr{C}_{D\otimes M_n(\C)}(X)$ be such that $[B\otimes \K]=0$ in $\bar{E}^1_D(X)$.
Then there is an isomorphism of continuous fields $\phi: B\otimes \K \stackrel{\cong}\longrightarrow C(X)\otimes D \otimes \K$. After conjugating $\phi$ by a unitary we may assume that $p:=\phi(1_B\otimes e_{11})\in C(X)\otimes D \otimes M_N(\C)$ for some integer $N\geq 1$.  It follows immediately that the projection $p$ has the desired properties. Conversely, if $B\cong p \left(C(X)\otimes D \otimes M_N(\C)\right)p$ then there is an isomorphism of continuous fields $B\otimes \K\cong C(X)\otimes D\otimes \K$ by \cite{paperL.G.Brown.Stable.Isom}.
We have thus shown that that $\theta([B]_{Br})=0$ iff and only if
$[B]_{Br}=0$.

We are now able to conclude that
$Br_D(X)$ is a group and that $\theta$ is injective by the general observation made earlier.
\end{proof}

 \begin{definition} \label{def:torsion} Let $D$ be  a strongly self-absorbing $C^*$-algebra.
 Let $A$ be
a locally trivial continuous field of $C^*$-algebras  with fiber $D \otimes \K$.  We say that $A$ is a \emph{torsion continuous field} if $A^{\otimes k}$ is isomorphic to a trivial field for some integer $k \geq 1$.
\end{definition}
\begin{corollary} \label{cor:torsion_delta}
Let $A$ be as in Theorem~\ref{thm:applications}.
Then $A$ is a torsion continuous field if and only if $\delta_0(A) \in H^1(X, K_0(D)^{\times}_+)$ is a torsion element and $\delta_k(A) = 0 \in H^{2k+1}(X,\Q)$ for all $k\geq 1$.
\end{corollary}
\begin{proof} Let $m\geq 1$ be an integer such that $m\delta_0(A) =0$. Then $\delta_0(A^{\otimes m})=0$.
We conclude by applying Theorem~\ref{thm:applications} to the orientable continuous field $A^{\otimes m}$.
\end{proof}

\section{Characteristic classes of the opposite continuous field}
Given a $C^*$-algebra $B$ denote by $B^{\rm op}$ the \emph{opposite $C^*$-algebra} with the same underlying Banach space and norm, but with multiplication given by $b^{\rm op} \cdot a^{\rm op} = (a \cdot b)^{\rm op}$. The \emph{conjugate $C^*$-algebra} $\overline{B}$ has the conjugate Banach space as its underlying vector space, but the same multiplicative structure. The map $a \mapsto a^*$ provides an isomorphism $B^{\rm op} \to \overline{B}$. Any automorphism $\alpha \in \Aut{B}$ yields in a canonical way automorphisms $\bar{\alpha} \colon \overline{B} \to \overline{B}$ and $\alpha^{\rm op} \colon B^{\rm op} \to B^{\rm op}$ compatible with $\ast \colon B^{\rm op} \to \overline{B}$. Therefore we have group isomorphisms $\theta \colon \Aut{B} \to \Aut{\overline{B}}$ and $\Aut{B} \to \Aut{B^{\rm op}}$. 
Note that $\alpha\in \Aut{B}$ is equal to $\theta(\alpha)$ when regarded as set-theoretic maps $B\to B$.
Given a locally trivial continuous field $A$ with fiber $B$, we can apply these operations fiberwise to obtain the locally trivial fields $A^{\rm op}$ and $\overline{A}$, which we will call the \emph{opposite} and the \emph{conjugate field}. They are isomorphic to each other and isomorphic to the conjugate and the opposite $C^*$-algebras of $A$.

A \emph{real form} of a complex C*-algebra $A$ is a real C*-algebra $A^{\R}$ such that $A\cong A^{\R}\otimes \C$.  A real form is not necessarily unique \cite{Boersema_and_comp:class} and not all C*-algebras admit real forms \cite{Phillips:opposite}.
If two C*-algebras $A$ and $B$ admit real forms  $A^{\R}$ and  $B^{\R}$, then $A^{\R}\otimes_{\R} B^{\R}$
is a real form of $A\otimes B$. 

\begin{example} All known strongly self-absorbing C*-algebras $D \in \mathcal{D}$ admit a real form.

Indeed, the real Cuntz algebras $\Cuntz{2}^{\R}$ and $\Cuntz{\infty}^{\R}$ are defined by the same
generators and relations as their complex versions. Alternatively $\Cuntz{\infty}^{\R}$ can be realized as follows.
Let $H_{\R}$ be a separable infinite dimensional real Hilbert space and let $\mathcal{F}^{\R}(H_{\R}) = \bigoplus_{n = 0}^{\infty} H_{\R}^{\otimes n}$ be the real Fock space associated to it. Every $\xi \in H_{\R}$ defines a shift operator $s_{\xi}(\eta) = \xi \otimes \eta$ and we denote the algebra spanned by the $s_{\xi}$ and their adjoints $s_{\xi}^*$ by $\Cuntz{\infty}^{\R}$. If $\mathcal{F}(H_{\R} \otimes \C)$ denotes the Fock space associated to the complex Hilbert space $H = H_{\R} \otimes \C$, then we have $\mathcal{F}^{\R} \otimes \C \cong \mathcal{F}(H)$. If we represent $\Cuntz{\infty}$ on $\mathcal{F}(H)$ using the above construction, then the map $s_{\xi} + i\,s_{\xi'} \mapsto s_{\xi + i\,\xi'}$ induces an isomorphism $\Cuntz{\infty}^{\R} \otimes \C \to \Cuntz{\infty}$. Likewise define $M_{\Q}^{\R}$ to be the infinite tensor product $M_2(\R) \otimes M_3(\R) \otimes M_4(\R) \otimes \dots$. Since $M_n(\C) \cong M_n(\R) \otimes \C$, we obtain an isomorphism $M_{\Q}^{\R} \otimes \C \cong M_{\Q}$ on the inductive limit. Let $\K^{\R}$ be the compact operators on $H_{\R}$ and $\K$ those on $H$, then we have $\K^{\R} \otimes \C \cong \K$. Thus, $M_{\Q} \otimes \Cuntz{\infty} \otimes \K$ is the complexification of the real $C^*$-algebra $M_{\Q}^{\R} \otimes \Cuntz{\infty}^{\R} \otimes \K^{\R}$.

The Jiang-Su  algebra $\ZZ$ admits a real form $\ZZ^{\R}$ which can be constructed in the same way as $\ZZ$.
Indeed, one constructs $\ZZ^{\R}$ as the inductive limit of a system 
\[ 
\cdots \to C([0,1], M_{p_nq_n}(\R)) \stackrel{\phi_{n}}{\longrightarrow} C([0,1], M_{p_{n+1}q_{n+1}}(\R))\to \cdots
\]
where the connecting maps $\phi_n$ are defined just as in the proof of \cite[Prop. 2.5]{Jiang_Su:paper}
with  only one modification. Specifically, one can choose the matrices $u_0$ and $u_1$ to be in the special orthogonal group 
$SO(p_nq_n)$ and this will ensure the existence of a continuous path $u_t$ in $O(p_nq_n)$ from $u_0$ to $u_1$ as required. 
\end{example}

If $B$ is the complexification of a real $C^*$-algebra $B^{\R}$, then a choice of isomorphism $B \cong B^{\R} \otimes \C$ provides an isomorphism $c \colon B \to \overline{B}$ via complex conjugation on $\C$. On automorphisms we have ${\rm Ad}_{c^{-1}} \colon \Aut{\overline{B}} \to \Aut{B}$. Let $\eta =  {\rm Ad}_{c^{-1}} \circ \theta \colon \Aut{B} \to \Aut{B}$. Now we specialize to the case $B =D \otimes \K$ with $D\in \mathcal{D}$ and study the effect of $\eta$ on homotopy groups, i.e.\ $\eta_* \colon \pi_{2k}(\Aut{B}) \to \pi_{2k}(\Aut{B})$. By \cite[Theorem 2.18]{paper:DadarlatP1} the groups $\pi_{2k+1}(\Aut{B})$ vanish.

Let $R$ be a commutative ring and denote by $\left[K^0(S^{2k}) \otimes R\right]^{\times}$  the group of units of the ring $K^0(S^{2k}) \otimes R$. Let $\left[K^0(S^{2k}) \otimes R\right]^{\times}_1$ be the kernel of the 
morphism of multiplicative groups $\left[K^0(S^{2k}) \otimes R\right]^{\times} \to R^{\times}.$ 
 This is the group of virtual rank $1$ vector bundles with coefficients in $R$ over $S^{2k}$. Let $c_S \colon K^0(S^{2k}) \to K^0(S^{2k})$ and $c_R \colon K_0(D) \to K_0(D)$ be the ring automorphisms induced by complex conjugation. 

\begin{lemma} \label{lem:KTheorySphere}
Let $D$ be a strongly self-absorbing $C^*$-algebra in the class $\mathcal{D}$, let $R = K_0(D)$ and let $k > 0$. There is an isomorphism $\pi_{2k}(\Aut{D \otimes \K}) \to \left[K^0(S^{2k}) \otimes R\right]^{\times}_1$ ($k>0$) such that the following diagram commutes
\[
	\xymatrix{
		\pi_{2k}(\Aut{D \otimes \K}) \ar[d] \ar[rr]^{\eta_*} && \pi_{2k}(\Aut{D \otimes \K}) \ar[d] \\
		\left[K^0(S^{2k}) \otimes R\right]^{\times}_1 \ar[rr]^{c_S \otimes c_R} && \left[K^0(S^{2k}) \otimes R\right]^{\times}_1
	}
\]
\end{lemma}

\begin{proof}
Observe that $\pi_{2k}(\Aut{D \otimes \K}) = \pi_{2k}(\uAut{D \otimes \K})$ (for $k>0$) and  $\uAut{D \otimes \K}$ is a path connected group, therefore $\pi_{2k}(\Aut{D \otimes \K}) = [S^{2k}, \uAut{D \otimes \K}]$. Let $e \in \K$ be a rank $1$ projection such that $c(1_D \otimes e) = 1_D \otimes e$. It follows from the proof of \cite[Theorem 2.22]{paper:DadarlatP1} that the map $\alpha \mapsto \alpha(1 \otimes e)$ induces an isomorphism $[S^{2k}, \uAut{D \otimes \K}] \to K_0(C(S^{2k}) \otimes D)^{\times}_1 = 1 + K_0(C_0(S^{2k} \setminus x_0) \otimes D)$. We have $\eta(\alpha)(1 \otimes e) = c^{-1}(\alpha ( c(1 \otimes e)))=c^{-1}(\alpha ( 1 \otimes e))$, i.e.\ the isomorphism intertwines $\eta$ and $c^{-1}$. Consider the following diagram of rings:
\[
	\xymatrix{
		K^0(S^{2k}) \otimes R \ar[d] \ar[rr]^{c_S \otimes c_R} && K^0(S^{2k}) \otimes R \ar[d] \\
		K_0(C(S^{2k}) \otimes D) \ar[rr]^{p \mapsto c^{-1}(p)} && K_0(C(S^{2k}) \otimes D)
	}
\]
The vertical maps arise from the K\"unneth theorem. Since $K_1(D) = 0$, these are isomorphisms. Since $c_S$ corresponds to the operation induced on $K_0(C(S^{2k}))$ by complex conjugation on $\K$, the above diagram commutes. 
\end{proof}

\begin{remark}\label{remark:conj} (i) If $D \in \mathcal{D}$  then $R=K_0(D) \subset \Q$ with $[1_D]=[1_{D^{\R}}]=1$. Thus  $c^{-1}(1_D) = 1_D$ and this shows that the above automorphism $c_R$ is trivial. 
The $K^0$-ring of the sphere is given by $K^0(S^{2k}) \cong \Z[X_k]/(X_k^2)$. The element $X_k$ is the $k$-fold reduced exterior tensor power of $H - 1$, where $H$ is the tautological line bundle over $S^2 \cong \C P^1$. Since $c_S$ maps $H-1$ to $1 - H$, it follows that $X_k$ is mapped to $-X_k$ if $k$ is odd and to $X_k$ if $k$ is even. We have $\left[K^0(S^{2}) \otimes R\right]^{\times}_1 = \{ 1 + t\,X_k\ |\ t \in R \} \subset R[X_k]/(X_k^2)$. Thus, $c_S$ maps $1 + t\,X_k$ to its inverse $1 - t\,X_k$ if $k$ is odd and acts trivially if $k$ is even.

(ii) By \cite[Theorem 2.18]{paper:DadarlatP1} there is an isomorphism $\pi_0(\Aut{D\otimes \K})\cong K_0(D)^{\times}_+$ given by $[\alpha]\mapsto [\alpha(1\otimes e)]$. Arguing as in Lemma~\ref{lem:KTheorySphere} we see that the action of $\eta$ on this groups is given by $c_R=\mathrm{id}$.
\end{remark}
\begin{theorem} \label{thm:char_and_opp}
Let $X$ be a compact metrizable space and let $A$ be a locally trivial continuous field with fiber $D \otimes \K$ for a strongly self-absorbing $C^*$-algebra $D\in \mathcal{D}$. Then we have for $k\geq  0$:
\[
	\delta_k(A^{\rm op}) = \delta_k(\overline{A}) = (-1)^k\,\delta_k(A) \in H^{2k+1}(X, \Q)\ .
\]
\end{theorem}

\begin{proof} \label{pf:char_and_opp}
Let $D^{\R}$ be a real form of $D$. 
The group isomorphism $\eta \colon \Aut{D\otimes \K} \to \Aut{D\otimes \K}$ induces an infinite loop map $B\eta \colon B\Aut{D\otimes \K} \to B\Aut{D\otimes \K}$, where the infinite loop space structure is the one described in \cite[Section 3]{paper:DadarlatP1}. 
 If $f \colon X \to B\Aut{D\otimes \K}$ is the classifying map of a locally trivial field $A$, then $B\eta \circ f$ classifies $\overline{A}$. Thus the induced map $\eta_*:E_{D}^1(X)\to E_{D}^1(X)$ has the property that
 $\eta_*[A]=[\overline{A}]$.

The unital inclusion $D^{\R}\to B^{\R}:=D^{\R} \otimes \Cuntz{\infty}^{\R}\otimes M_{\Q}^{\R}$ induces a commutative diagram
\[
	\xymatrix{
		\Aut{D\otimes \K}\ar[r]^{\eta} \ar[d] & \Aut{D\otimes \K}\ar[d]\\
		\Aut{B\otimes \K}\ar[r]^{\eta} & \Aut{B\otimes \K}
	}
\]
with $B:=B^{\R}\otimes \C$. 
From this we obtain a commutative  diagram
\[
	\xymatrix{
		E^1_D(X)\ar[r]^{\eta_*} \ar[d]_{\delta} & E^1_D(X)\ar[d]^{\delta}\\
		E^1_B(X)\ar[r]^{\eta_*} & E^1_B(X)
	}
\]
As explained earlier, $B\cong M_{\Q} \otimes \Cuntz{\infty}$.
Recall that
 $E^1_{M_{\mathbb{Q}}\otimes \mathcal{O}_\infty}(X)\cong H^1(X,\mathbb{Q}^\times)\oplus \bigoplus_{k \geq 1} H^{2k+1}(X,\mathbb{Q})$.
 By Lemma \ref{lem:KTheorySphere} and Remark~\ref{remark:conj}(i) the effect of $\eta$ on $H^{2k+1}(X, \pi_{2k}(\Aut{B})) \cong H^{2k+1}(X, \Q)$ is given by multiplication with $(-1)^k$ for $k>0$.
 By Remark~\ref{remark:conj}(ii) $\eta$ acts trivially on $H^1(X,\pi_0(\Aut{B}))=H^1(X,\Q^{\times})$.
\end{proof}
\begin{example}\label{remark:brauer_op}
Let $\ZZ$ be the Jiang-Su algebra. We will show that in general the inverse of an element  in the Brauer group $Br_{\ZZ}(X)$ is not represented  by the class of the opposite algebra.
Let $Y$ be the space obtained by attaching a disk to a circle by a degree three map and let $X_n=S^n\wedge Y$
be $n^{th}$ reduced suspension of $Y$. Then $E_{\ZZ}^1(X_3)\cong K^0(X_2)^{\times}_{+}\cong 1+\widetilde{K}^0(X_2)$
by \cite[Thm.2.22]{paper:DadarlatP1}. Since this is a torsion group, $Br_{\ZZ}(X_3)\cong E_{\ZZ}^1(X_3)$
by Theorem~\ref{thm:Brauer_Serre}. Using the K\"unneth formula, 
$Br_{\ZZ}(X_3)\cong 1+\widetilde{K}^0(S^2)\otimes \widetilde{K}^0(Y)\cong 1+ \Z/3$.
Reasoning as in Lemma~\ref{lem:KTheorySphere} with $X_2$ in place of $S^{2k}$, we identify the map $\eta_*:E_{\ZZ}^1(X_3) \to E_{\ZZ}^1(X_3)$ with the map $K^0(X_2)^{\times}_{+}\to K^0(X_2)^{\times}_{+}$ that sends the class $x=[V_1]-[V_2]$ to   $\overline{x}=[\overline{V}_1]-[\overline{V}_2]$, where 
$\overline{V}_i$ is the complex conjugate bundle of $V_i$. If $V$ is complex vector bundle, and $c_1$ is the first Chern class, $c_1(\overline{V})=-c_1(V)$ by \cite[p.206]{book:Karoubi}. Since conjugation is compatible with
the K\"unneth formula, we deduce  that $x= \overline{x}$ for $x\in  K^0(X_2)^{\times}_{+}$.
Indeed, if 
 $\beta\in \widetilde{K}^0(S^2)$, $y\in\widetilde{K}^0(Y)$ and $x=1+\beta y$, then $\overline{x}=1+(-\beta)(-y)=x$.
 Let $A$ be a continuous field over $X_3$ with fibers $M_N(\ZZ)$ such that
$[A]_{Br}=1+\beta y$ in $Br_{\ZZ}(X_3)\cong 1+\widetilde{K}^0(S^2)\otimes \widetilde{K}^0(Y)\cong 1+ \Z/3$, where $\beta$ a generator of $\widetilde{K}^0(S^2)$ and $y$ is a generator of  $\widetilde{K}^0(Y)$. 
Then $[\overline{A}]_{Br}=1+(-\beta)(-y)=[A]_{Br}$ and hence 
$$[\overline{A}\otimes_{C(X_3)} A]_{Br}=(1+\beta y)^2=1+2\beta y\neq 1.$$
\end{example}
\begin{corollary} \label{lem:A_Aop_torsion}
Let $X$ be a compact metrizable space and let $A$ be a  locally trivial continuous field with fiber $D \otimes \K$ with $D$ in the class $\mathcal{D}$. If $H^{4k + 1}(X, \Q) =  0$ for all $k\geq 0$, then there is an $N \in \N$ such that 
\[
	(A \otimes_{C(X)} A^{\rm op})^{\otimes N} \cong C(X, D \otimes \K)\ .
\]
\end{corollary}

\begin{proof} \label{pf:A_Aop_torsion}
If $H^{4k+1}(X,\Q) = 0$, then $\delta_{2k}(A \otimes_{C(X)} A^{\rm op}) = 0$ for all $k\geq 0$. Moreover,
\newline $\delta_{2k+1}(A \otimes_{C(X)} A^{\rm op}) = \delta_{2k+1}(A) - \delta_{2k+1}(A) = 0$. The statement follows from Corollary \ref{cor:torsion_delta}.
\end{proof}

\renewcommand*{\bibfont}{\footnotesize}
\bibliographystyle{plain}

\end{document}